\documentclass[twoside, 11pt]{article}
\usepackage{amssymb, amsmath, mathrsfs, amsthm}
\usepackage{graphicx}
\usepackage{color}
\usepackage[top=3cm, bottom=3cm, left=3cm, right=3cm]{geometry}
\usepackage{float, caption, subcaption}

\input{epsf}

\newcommand{\bd}{\begin{description}}
\newcommand{\ed}{\end{description}}
\newcommand{\bi}{\begin{itemize}}
\newcommand{\ei}{\end{itemize}}
\newcommand{\be}{\begin{enumerate}}
\newcommand{\ee}{\end{enumerate}}
\newcommand{\beq}{\begin{equation}}
\newcommand{\eeq}{\end{equation}}
\newcommand{\beqs}{\begin{eqnarray*}}
\newcommand{\eeqs}{\end{eqnarray*}}

\newcommand{\rn}[1]{{\color{red} #1}}

\definecolor{DarkGreen}{rgb}{0.2, 0.6, 0.3}
\newcommand{\gn}[1]{{\color{DarkGreen} #1}}

\newtheorem{theorem}{Theorem}[section]

\newtheorem{lemma}{Lemma}[section]

\newtheorem{case}{Case}

\newtheorem{claim}{Claim}

\newtheorem{fact}{Fact}
\newtheorem{proposition}{Proposition}[section]

\begin{document}
\title{\textbf{Gallai Ramsey number for double stars} \footnote{Supported by the National
Science Foundation of China (Nos. 11601254, 11551001, 11161037,
61763041, 11661068, and 11461054) and the Science Found of Qinghai
Province (Nos.  2016-ZJ-948Q, and 2014-ZJ-907) and the  Qinghai Key
Laboratory of Internet of Things Project (2017-ZJ-Y21).} }

\author{
Gyula O.H. Katona \footnote{Alfred Renyi Institute of Mathematics, Hungarian Academy of Sciences, Budapest Reaaltanoda utca 13-15, 1053},
Colton Magnant\footnote{Department of Mathematics, Clayton State University, Morrow, GA, 30260, USA. {\tt dr.colton.magnant@gmail.com}} \footnote{Academy of Plateau Science and Sustainability, Xining, Qinghai 810008, China},
Yaping Mao \footnote{School of Mathematics and Statistis, Qinghai Normal University, Xining, Qinghai 810008, China. {\tt maoyaping@ymail.com}} \footnotemark[4],
Zhao Wang\footnote{Corresponding Author - College of Science, China Jiliang University, Hangzhou 310018, China. {\tt wangzhao@mail.bnu.edu.cn}}}
\date{}
\maketitle

\begin{abstract}
Given a graph $G$ and a positive integer $k$, the \emph{Gallai-Ramsey number} is defined to be the minimum number of vertices $n$ such that any $k$-edge coloring of $K_n$ contains either a rainbow (all different colored) copy of $G$ or a monochromatic copy of $G$. In this paper, we obtain general upper and lower bounds on the Gallai-Ramsey numbers for double stars $S(n,m)$, where $S(n,m)$ is the graph obtained from the union of two stars $K_{1,n}$ and $K_{1,m}$ by adding an edge between their centers. We also provide the sharp result in some cases.
\end{abstract}

\section{Introduction}

In this work, we consider only colorings of the edges of graphs. A coloring of a graph $G$ is called \emph{rainbow} if no two edges in $G$ have the same color.

Edge colorings of complete graphs that contain no rainbow triangle have very interesting and somewhat surprising structure. In 1967, Gallai \cite{Gallai} examined this structure under the guise of transitive orientations (a translation of this paper is available in \cite{RamirezReed}). In honor of Gallai's result (Theorem~\ref{Thm:G-Part}), such colorings are are called \emph{Gallai-colorings}. The result was restated in \cite{GyarfasSimonyi} in the terminology of graphs and it can also be traced back to \cite{CameronEdmonds}. For the following statement, a partition is called trivial if it has only one part.

\begin{theorem}[\cite{Gallai, GyarfasSimonyi,CameronEdmonds}]\label{Thm:G-Part}
In any coloring of a complete graph containing no rainbow triangle, there exists a nontrivial partition of the vertices (called a Gallai-partition) such that there are at most two colors on the edges between the parts and only one color on the edges between each pair of parts.
\end{theorem}

The induced subgraph of a Gallai colored complete graph constructed by selecting a single vertex from each part of a Gallai partition is called its \emph{reduced graph}. By Theorem~\ref{Thm:G-Part}, the reduced graph is a $2$-colored complete graph.

Given two graphs $G$ and $H$, let $R(G, H)$ denote the $2$-color Ramsey number for finding a monochromatic $G$ or $H$, that is, the minimum number of vertices $n$ needed so that every red-blue coloring of $K_{n}$ contains either a red copy of $G$ or a blue copy of $H$. As an extension, let $R_{k}(H)$ denote the $k$-color Ramsey number for finding a monochromatic copy of $H$, that is, the minimum number of vertices $n$ so that every $k$-coloring of $K_{n}$ contains a monochromatic copy of $H$. Although the reduced graph of a Gallai partition uses only two colors, the original Gallai-colored complete graph could certainly use more colors. With this in mind, we consider the following generalization of the Ramsey numbers. Given two graphs $G$ and $H$, the \emph{$k$-colored Gallai-Ramsey number} $gr_k(G:H)$ is defined to be the minimum integer $m$ such that every $k$-coloring of the complete graph on $m$ vertices contains either a rainbow copy of $G$ or a monochromatic copy of $H$. With the additional restriction of forbidding the rainbow copy of $G$, it is clear that $gr_k(G:H)\leq R_k(H)$ for any graph $G$.

Although the concept of forbidding a rainbow triangle and looking for monochromatic subgraphs can be traced back to \cite{GyarfasSimonyi} and beyond, the study of Gallai-Ramsey numbers in their current form comes from \cite{FMO14}, where the authors used Theorem~\ref{Thm:G-Part} to produce sharp Gallai-Ramsey numbers for several small graphs along with some general bounds for paths and cycles. The subject has also been expanded through several other publications including but not limited to \cite{FM11, GSSS10}. In particular, in \cite{GSSS10}, the following general behavior of Gallai-Ramsey numbers was established.

\begin{theorem}[\cite{GSSS10}]\label{Thm:Dichotomy}
Let $H$ be a fixed graph with no isolated vertices. If $H$ is bipartite and not a star, then $gr_{k}(K_{3} : H)$ is linear in $k$. If $H$ is not bipartite, then $gr_{k}(K_{3} : H)$ is exponential in $k$.
\end{theorem}

We refer the interested reader to \cite{MR1670625} for a dynamic survey of small Ramsey numbers and \cite{FMO14} for a dynamic survey of rainbow generalizations of Ramsey theory, including topics like Gallai-Ramsey numbers.

The \emph{double star} $S(n,m)$ for integers $n\geq m\geq 0$ is the graph obtained from the union of two stars $K_{1,n}$ and $K_{1,m}$ by adding the edge between their centers. In this work, we prove bounds on the Gallai-Ramsey number of all double stars, along with sharp results for several small double stars.

One of the tools we will use is the following by Grossman et al. \cite{GrossmanHararyKlawe}, who obtained the exact value of classical Ramsey number of double stars in most cases.

\begin{theorem}[\cite{GrossmanHararyKlawe}]\label{Lem:RamseyS(n,m)}
$$
R(S(n,m),S(n,m)) = \begin{cases}
\max\{2n+1,n+2m+2\} & \text{ if $n$ is odd and $m\leq 2$}\\
\max\{2n+2,n+2m+2\} & \text{ if $n$ is even or $m\geq 3$,}\\
~ & \text{ and $n\leq \sqrt{2}m$ or $n\geq 3m$.}\\
\end{cases}
$$
\end{theorem}

This result was recently extended by Norin et al.~\cite{Norin} in the following result.

\begin{theorem}[\cite{Norin}]
For positive integers $m$ and $n$ with $m \leq n \leq 1.699(m + 1)$,
$$
R(S(n, m), S(n, m)) \leq n + 2m + 2.
$$
\end{theorem}

In this paper, we first obtain the exact value of the Gallai Ramsey number of double stars under an additional assumption that $n\geq 6m+7$.

\begin{theorem}\label{Thm:Ramseyn-24m}
Let $n,m,k$ be three integers with $m\geq 1$, $k\geq 3$ and $n\geq 6m+7$. Then
$$
gr_{k}(K_{3} : S(n,m))=\begin{cases}
5\cdot\frac{n}{2} + m(k-3)+1 & \text{ if $n$ is even,}\\
5\cdot\frac{n-1}{2} + m(k-3)+ 2 & \text{ if $n$ is odd.}
\end{cases}
$$
\end{theorem}

Next, we obtain the upper and lower bounds for Gallai Ramsey number of all double stars. 

\begin{theorem}\label{Thm:Ramseynm}
Let $n,m,k$ be three integers with $1 \leq m \leq n \leq 6m + 6$, $k\geq 3$. Then
$$
gr_{k}(K_{3} : S(n,m))\leq \begin{cases}
\max\{5\cdot\frac{n+2}{2},2n+6m+7\} + m(k-3)+1 & \text{ if $n$ is even,}\\
\max\{5\cdot\frac{n+1}{2},2n+6m+7\}+ m(k-3)+ 2 & \text{ if $n$ is
odd.}
\end{cases}
$$
and
$$
gr_{k}(K_{3} : S(n,m))\geq \begin{cases}
\max\{5\cdot\frac{n}{2},n+3m+1\} + m(k-3)+1 & \text{ if $n$ is even,}\\
\max\{5\cdot\frac{n-1}{2},n+3m+1\}+ m(k-3)+ 2 & \text{ if $n$ is
odd.}
\end{cases}
$$
\end{theorem}

The outline of the paper is as follows. Section~\ref{Sec2} contains the exact value of Gallai-Ramsey number for $S(n,m)$ under the condition $n\geq 6m+5$. We provide the upper and lower bounds of Gallai-Ramsey number for all double stars in Section~\ref{Sec3}.

\section{Proof of Theorem \ref{Thm:Ramseyn-24m}}\label{Sec2}

We first give the lower bound on the Gallai-Ramsey number for
$S(n,m)$.
\begin{lemma}\label{Lemma:k3Lower}
Let $n,m$ be two integers with $n\geq m$. Then
$$
gr_{3}(K_{3} : S(n,m))\geq \begin{cases}
5\cdot\frac{n}{2} +1 & \text{ if $n$ is even,}\\
5\cdot\frac{n-1}{2} + 2 & \text{ if $n$ is odd.}
\end{cases}
$$
\end{lemma}

\begin{proof}
We prove this result by inductively constructing a coloring of
$K_{\ell}$ where
$$
\ell = \begin{cases}
5\cdot\frac{n}{2}  & \text{ if $n$ is even,}\\
5\cdot\frac{n-1}{2} + 1 & \text{ if $n$ is odd.}
\end{cases}
$$
which contains no rainbow triangle and no monochromatic copy of
$S(n,m)$. Suppose this coloring uses colors red, blue and green. For
even $n$, we let $G_{1}$ be a red complete graph on $\frac{n}{2}$
vertices. We construct $G_3$ by making five copies of $G_{1}$ and
inserting edges of blue and green between the copies to form a
blow-up of the unique $2$-colored $K_{5}$ which contains no
monochromatic triangle.

For odd $n$, we let $G_{1}$ be a red complete graph on
$\frac{n-1}{2}$ vertices, and let $G'_{1}$ be a red complete graph on
$\frac{n+1}{2}$ vertices. We construct $G_3$ by making four copies
of $G_{1}$ and one copy of $G'_{1}$ and inserting edges of blue and
green between the copies to form a blow-up of the unique $2$-colored
$K_{5}$ which contains no monochromatic triangle.

This coloring clearly contains no rainbow triangle and, since there
is no vertex of degree at least $n + 1$, there
can be no monochromatic copy of $S(n,m)$, completing the
construction.
\end{proof}

Next, we give the upper bound on the Gallai-Ramsey number for
$S(n,m)$ with $3$ colors.
\begin{proposition}\label{Thm:RamseyS(n,m)k3}
Let $n,m$ be two integers with $m\geq 1$ and $n\geq 6m+7$. For the
double star $S(n,m)$, we have
$$
gr_{3}(K_{3} : S(n,m)) = \begin{cases}
5\cdot\frac{n}{2} +1 & \text{ if $n$ is even,}\\
5\cdot\frac{n-1}{2} + 2 & \text{ if $n$ is odd.}
\end{cases}
$$
\end{proposition}

\begin{proof}
Let $G$ be a coloring of $K_\ell$ where
$$
\ell=\begin{cases}
5\cdot\frac{n}{2} +1 & \text{ if $k$ is even,}\\
5\cdot\frac{n-1}{2} + 2 & \text{ if $k$ is odd.}
\end{cases}
$$

Since $G$ is a Gallai coloring, it follows from Theorem~\ref{Thm:G-Part} there is a Gallai partition of $V(G)$. Suppose red and
blue are the two colors appearing in the Gallai partition and let green be the other color. Let
$H_1,H_2,\ldots,H_t$ be the parts in this partition and choose such
a partition where $t$ is as small as possible. Without loss of generality, suppose these parts are labeled so that $|H_{1}| \geq |H_{2}| \geq \dots \geq |H_{t}|$.

\begin{claim}\label{Claim1}
For each part $H_i$ of $K_{\ell}$, $|H_i|\geq
\lfloor\frac{n}{2}\rfloor$ or $|H_i|\leq m$.
\end{claim}
\begin{proof}
Suppose that there exists one part, say $H_j$, and $m+1\leq
|H_j|\leq \lfloor\frac{n}{2}\rfloor-1$. Let $A$ be the set of parts
with red edges to $H_j$, and $B$ be the set of parts with blue edges
to $H_j$. Without loss of generality, $|A|\geq |B|$. Since
$|A|+|B| \geq \ell - (\lfloor \frac{n}{2} \rfloor - 1) > 2n$, it follows that one of these sets has order at least $n + 1$, say $|A|\geq n+1$. Then $A \cup H_{j}$ contains a red copy of $K_{n+1,m+1}$, which contains a red copy of $S(n,m)$, a contradiction.
\end{proof}

Let $r$ be the number of ``large'' parts $H_{i}$ of the Gallai partition with $|H_i|\geq \frac{n-1}{2}$, so $|H_{r}| \geq \frac{n - 1}{2}$ but $H_{r + 1} < \frac{n - 1}{2}$. Otherwise call the parts ``small''. We break the proof into cases based on the values of $t$ and $r$.

\setcounter{case}{0}
\begin{case}
$2\leq t\leq 3$.
\end{case}

By the minimality of $t$, we may assume
$t=2$, say with corresponding parts $H_1$ and $H_2$. Suppose all
edges between $H_1$ and $H_2$ are blue. If $|H_2|\geq m+1$, then to avoid a
monochromatic $S(n,m)$, we have $|H_1|\leq n$ and $|H_2|\leq n$, and
hence $|G|=|H_1|+|H_2|<\ell$, a contradiction. Thus, we may assume $|H_2|\leq m$.

\begin{claim}\label{Claim2}
For any vertex $x$ of $H_1$, there are at most $m$ incident blue edges in $G$.
\end{claim}

\begin{proof}
Assume, to the contrary, that there exists a vertex $u$ in $H_1$ with at least $m+1$ incident blue edges.
Choose any vertex $w\in H_2$ and $u_1,u_2,\ldots,u_m\in N_G^{blue}(u) \setminus \{w\}$ and
$v_1,v_2,\ldots,v_n\in H_1-N_G^{blue}(u)$. Then the blue edges in
$\{uu_i\,|\,1\leq i\leq m\}\cup \{wv_j\,|\,1\leq j\leq n\} \cup \{uw\}$ form a
blue copy of $S(n,m)$, a contradiction.
\end{proof}

From Claim \ref{Claim2}, for any vertex $v$ of $H_1$, the total number of green or red edges incident to $v$ in $H_1$ is at least
$$
5\cdot \frac{n-1}{2}+2-m-1\geq 2(n+m+1),
$$
since $n\geq 6m+7$. Choose $v \in H_{1}$. Without loss of generality, we assume that the
number of edges incident to $v$ in green is more than the number of edges incident to $v$ in red. Then there are at least $n+m+1$ green edges incident with $v$, and call the corresponding vertices $v_{1},v_{2},\ldots,v_{n+m+1}$.

\begin{claim}\label{Claim3}
For any $v_{i}$ with $1\leq i\leq n+m+1$, the number of green edges incident to $v_{i}$ in $G$ is at most $m$.
\end{claim}
\begin{proof}
Assume, to the contrary, that there exists a vertex $v_j$ with at least $m+1$ incident green edges. Choose $X\subseteq N_G^{green}(v_j) \setminus \{v\}$ with $|X|=m$. Then the edges
from $v_j$ to $X$ and the edge $vv_j$ and the edges from $v$ to
$\{vv_i\,|\,1\leq i\leq n+m+1\}-X-v_j$ form a green copy of $S(n,m)$, a
contradiction.
\end{proof}

From Claims~\ref{Claim2} and~\ref{Claim3}, for any $v_{i}$ with $1\leq i\leq n+m+1$, the number of red edges incident to $v_{i}$ within $H_1$ is at least $2n$. Since $v$ has at least $n + m + 1$ incident green edges, there exists a red edge within $G[N_{G}^{green}(v)]$. We choose two vertices, say $v_i,v_{j}$, such that the
edge $v_iv_{j}$ is red. Then the red edges incident to $v_i,v_{j}$
form a red copy of $S(n,m)$, a contradiction.

\begin{case}
$t\geq 4$ and $r\geq 3$.
\end{case}

\begin{claim}\label{Claim4}
Any choice of $3$ large parts in $\mathscr {H}=\{H_1,\ldots,H_r\}$ do not
form a monochromatic triangle in the reduced graph.
\end{claim}

\begin{proof}
Assume, to the contrary, that there exist $3$ parts in $\mathscr
{H}=\{H_1,\ldots,H_r\}$ that form a monochromatic triangle in the reduced
graph, say red triangle $H_1H_2H_3$. To avoid a $2$-partition in blue (thereby contradicting the minimality of $t$), there
exists at least one red edge from $\bigcup_{i=1}^{3}H_i$ to
$\bigcup_{i=4}^{t}H_i$, that is, there exists a vertex in
$\bigcup_{i=4}^{t}H_i$, say $v$, and a part $H_1$ or $H_2$ or $H_3$,
say $H_1$, such that all edges from $v$ to $H_1$ are red.

First suppose $n$ is even. Then we choose a vertex $u\in H_1$. Then $m$ of the edges from $v$ to
$H_1-u$, the edge $uv$, and $n$ of the edges from $u$ to $H_2\cup H_3$ form a red copy of $S(n,m)$, a contradiction.

Next suppose $n$ is odd. Suppose further that there exist two vertices in $x,y\in
\bigcup_{i=4}^{t}H_i$ and a part in $\{H_1,H_2,H_3\}$, say $H_1$,
such that the edges from $\{x,y\}$ to $H_1$ are red. Then we choose a vertex
$u\in H_1$. Then $n$ of the edges from $u$ to $H_2\cup H_3\cup \{y\}$,
$m$ of the edges from $x$ to $H_1 \setminus \{u\}$, and the edge from $u$ to $x$ form a red copy of $S(n,m)$, a contradiction.

We suppose now that for each part $H_{i}$ (for $i$ with $1 \leq i \leq 3$), there is at most one vertex $x_{i}$ in $\bigcup_{i=4}^{t}H_i$, with red edges to $H_{i}$.
We claim that
$|H_i|\leq \frac{n+1}{2}$ for $i=1,2,3$. Assume, to the contrary,
that there exists a part, say $H_1$, such that $|H_1|\geq
\frac{n+3}{2}$. Then we arbitrarily choose vertices $u\in H_1$ and $v\in H_2$. Then the edges from $u$ to $H_2 \setminus \{v\}$, the edges from $v$ to $(H_1\setminus \{u\})\cup H_3$ form a red copy of $S(n,m)$, a contradiction. This means that $\sum_{i=1}^{3}|H_i|\leq \frac{3n+3}{2}$, and so $\sum_{i=4}^{t}|H_i|-3\geq \frac{5(n-1)}{2}-\frac{3n+3}{2}\geq m+4-3=m+1$. Then there is a blue copy of $K_{n+1,m+1}$ on edges between $\bigcup_{i=4}^{t}H_i$ and $\bigcup_{i=1}^{3}H_i$ which contains a blue copy of $S(n, m)$, a contradiction.
\end{proof}

From Claim \ref{Claim4}, any choice of $3$ parts in $\mathscr
{H}=\{H_1,\ldots,H_r\}$ form a monochromatic $P_3$ in the reduced
graph. Without loss of generality (ignoring previous assumptions on the relative sizes between $H_{1}$, $H_{2}$, and $H_{3}$), let $P_3=H_1H_2H_3$ be red.

\begin{claim}\label{Claim5}
$n - 1 \leq |H_1|+|H_3|\leq n$ and $|H_2|\leq n$.
\end{claim}
\begin{proof}
If either $|H_1|+|H_3|\geq n+1$ or $|H_2|\geq n+1$, then the subgraph induced on the red edges from $H_{2}$ to $H_{1} \cup H_{3}$ contains a red copy of $K_{m + 1, n + 1}$, which contains a red copy of $S(n, m)$, a contradiction.
\end{proof}

From Claim \ref{Claim5}, we have $\sum_{i=4}^{t}|H_i| \geq \ell - 2n \geq \lfloor
\frac{n}{2}\rfloor$. Furthermore, we have the following claim.

\begin{claim}\label{Claim6}
$\lfloor \frac{n}{2}\rfloor\leq \sum_{i=4}^{t}|H_i|\leq n+1$.
\end{claim}

\begin{proof}
Assume, to the contrary, that $\sum_{i=4}^{t}|H_i|\geq n+2$. To avoid a blue copy of $K_{n+1,m+1}$, there are at most $n$ vertices in $\bigcup_{i=4}^{t}H_i$ with blue edges to $H_2$. Then there are at least $2$ vertices say $v_{1}, v_{2}$ in $\bigcup_{i=4}^{t}H_i$ with red edges to $H_2$. Since the edges from $H_2$ to $H_1\cup H_3 \cup \{v_{1}, v_{2}\}$ are all red, this contains a red copy of $K_{m + 1, n + 1}$, which contains a red copy of $S(n, m)$, a contradiction.
\end{proof}

From the proof of Claim~\ref{Claim6}, there is at most one vertex in $\bigcup_{i=4}^{t}H_i$ with red edges to $H_2$, and hence there are at least $\lfloor \frac{n}{2}\rfloor-1$ vertices in $\bigcup_{i=4}^{t}H_i$ with blue edges to $H_2$.

\begin{claim}\label{Claim7}
$|H_1|+|H_2|\leq n$ and $|H_2|+|H_3|\leq n$.
\end{claim}

\begin{proof}
By symmetry, we only prove $|H_1|+|H_2|\leq n$. Assume, to the
contrary, that $|H_1|+|H_2|\geq n+1$. Let $A$ be the set of parts (other than $H_{1}$ and $H_{3}$) with red edges to $H_2$ and $B$ be the set of parts with blue edges
to $H_2$. Since there is at most one vertex in $\bigcup_{i=4}^{t}H_i$ with red edges to $H_2$, it follows that $|A|\leq 1$, and hence $|B|\geq \lfloor \frac{n}{2}\rfloor-1$. In order to avoid a blue copy of $S(n,m)$, the edges from $B$ to $H_1$ must all be red, and hence $|B|+|H_2|\leq n$ (by Claim~\ref{Claim5}). Then $\sum_{i=1}^{t}|H_i|= |H_1|+|H_2|+|H_3|+|A|+|B|\leq 2n+|A|\leq 2n+1<|G|$, a contradiction.
\end{proof}

Suppose that $n$ is even. From Claims~\ref{Claim5} and~\ref{Claim7}, $|H_1|+|H_2|+|H_3|\leq \frac{3n}{2}$. Recall that for each $i$ with $i=1,2,3$, we have $|H_i|\geq \frac{n}{2}$. This means that $|H_1|=|H_2|=|H_3|=\frac{n}{2}$, and hence $\sum_{i=4}^{n}|H_i|=n+1$. Suppose that there is a vertex $v_{1}$ in $\bigcup_{i=4}^{n}H_i$ with red edges to $H_2$. Since the edges from $H_{2}$ to $H_{1} \cup H_{3} \cup \{v_{1}\}$ are all red, this contains a red copy of $K_{m + 1, n + 1}$, which contains a red copy of $S(n, m)$, a contradiction.
This means that all edges from $\bigcup_{i=4}^{n}H_i$ to $H_2$ are blue. Since $|H_2|=\frac{n}{2} \geq m + 1$ and $\sum_{i=4}^{n}|H_i|=n+1$, it follows that there is a blue copy of $K_{m+1,n+1}$ between $H_2$ and $\bigcup_{i=4}^{n}H_i$ and therefore a blue copy of $S(n, m)$, a contradiction.

Thus, assume that $n$ is odd. From Claims \ref{Claim5} and \ref{Claim7}, $|H_1|+|H_2|+|H_3|\leq \frac{3n-1}{2}$. Recall that for each $i$ with $i=1,2,3$, we have $|H_i|\geq \frac{n-1}{2}$. Then either $|H_1|=|H_2|=|H_3|=\frac{n-1}{2}$ or $|H_1|=\frac{n+1}{2}$ and $|H_2|=|H_3|=\frac{n-1}{2}$.

First suppose $|H_1|=|H_2|=|H_3|=\frac{n-1}{2}$, then $\sum_{i=4}^{n}|H_i|=n+1$. To avoid a red copy of $S(n,m)$, the number of vertices in $\bigcup_{i=4}^{n}H_i$ with red edges to $H_2$ is at most $1$. On the other hand, to avoid a blue copy of $K_{m+1,n+1}$, the number of vertices in $\bigcup_{i=4}^{n}H_i$ with red edges to to $H_2$ is at least (and therefore exactly) $1$. Let $v$ be this vertex (note that $v$ is in a part of order $1$), say with $v \in H_{t}$, and let $u \in H_{2}$. To avoid a red copy of $S(n,m)$, all edges from $v$ to either $H_1$ or $H_3$ must be blue. Without loss of generality, we suppose that the edges from $v$ to $H_1$ are blue.

Since $v$ has only red or blue edges, $v$ has at least $n+m+1$ edges in either red or blue. If $v$ has at least $n + m + 1$ incident red edges, then let $u$ be an arbitrary vertex in $H_{2}$. Then the red edges from $u$ to $m$ vertices of $H_{1}$, the edge from $u$ to $v$, and $n$ remaining red edges incident to $v$ form a red copy of $S(n, m)$, a contradiction. Otherwise suppose $v$ has at least $n + m + 1$ incident blue edges and let $u$ be an arbitrary vertex in $H_{1}$. Then the blue edges from $u$ to $m$ vertices of $H_{3}$, the edge from $u$ to $v$, and $n$ remaining blue edges incident to $v$ form a blue copy of $S(n, m)$, a contradiction.

Thus, suppose $|H_1|=\frac{n+1}{2}$ and $|H_2|=|H_3|=\frac{n-1}{2}$. Then this means that $\sum_{i=4}^{t}|H_i|=n$. Suppose first that the larger set $H_{1}$ is in the middle of the red path of the reduced graph, that is, $P_3=H_2H_1H_3$.

\begin{claim}\label{Claim8}
All edges from $H_1$ to $\bigcup_{i=4}^{n}H_i$ are blue.
\end{claim}

\begin{proof}
Assume, to the contrary, that there exists a vertex $v$ with all red edges from $v$ to $H_{1}$. As before, $v$ must be in a part of order $1$. Also again the edges from $v$ to at least one of $H_2$ or $H_3$ are all blue. Without loss of generality, we suppose that the edges from $v$ to $H_2$ are blue. As above, $v$ has at least $n + m + 1$ edges in either red or blue. If these edges are red, then choosing an arbitrary vertex $u \in H_{1}$, the red edges from $u$ to $m$ vertices of $H_{2}$, the edge $uv$, and $n$ remaining red edges incident to $v$ form a red copy of $S(n, m)$. Otherwise if these edges are blue, then choosing an arbitrary vertex $u \in H_{2}$, the blue edges from $u$ to $m$ vertices of $H_{3}$, the edge $uv$, and $n$ remaining blue edges incident to $v$ form a blue copy of $S(n, m)$, a contradiction.
\end{proof}

From Claim~\ref{Claim8}, the edges from $\bigcup_{i=4}^{t}H_i$ to $H_1$ are blue. Let $A$ be the set of parts in $\bigcup_{i=4}^{t}H_i$ with red edges to $H_2$ and $B$ be the set of parts in $\bigcup_{i=4}^{t}H_i$ with blue edges to $H_2$. To avoid a red copy of $S(n,m)$ centered at a pair of vertices in $H_{2}$ and $H_{1}$, we must have $|A|\leq \frac{n-1}{2}$ since otherwise we could find a red copy of $S(n, m)$ using $m$ neighbors in $H_{3}$ of a vertex in $H_{1}$ and $n$ neighbors in $H_{1} \cup A$ of a vertex in $H_{2}$. Hence we have $|B|\geq \frac{n+1}{2}$. To avoid a blue copy of $S(n,m)$ (centered at vertices in $B$ and $H_{2}$), the edges from $B$ to $H_3$ must be red. Then $m$ neighbors in $H_{2}$ of a vertex in $H_{1}$ along with $n$ neighbors in $B \cup H_{1}$ of a vertex in $H_{3}$ form a red copy of $S(n,m)$, a contradiction.

Finally suppose the larger set is at one end of the red path in the reduced graph, that is, $P_3=H_1H_2H_3$. Recall that $|H_1|=\frac{n+1}{2}$, $|H_2|=|H_3|=\frac{n-1}{2}$, and $\sum_{i=4}^{t}|H_i|=n$. To avoid a red copy of $S(n,m)$ centered at vertices in $H_{1}$ and $H_{2}$, the edges from $\bigcup_{i=4}^{t}H_i$ to $H_2$ must be blue. Let $A$ be the set of parts in $\bigcup_{i=4}^{t}H_i$ with red edges to $H_3$ and $B$ be the set of parts in $\bigcup_{i=4}^{t}H_i$ with blue edges to $H_3$. To avoid a red copy of $S(n,m)$, we again have $|A|\leq \frac{n+1}{2}$ and to avoid a blue copy of $S(n,m)$, we have $|B|\leq \frac{n-1}{2}$, and hence $|A|=\frac{n+1}{2}$ and $|B|=\frac{n-1}{2}$. To avoid a blue copy of $S(n,m)$ centered at vertices in $H_{1}$ and $B$, the edges from $B$ to $H_1$ are red. This structire implies the following claim.

\begin{claim}\label{Claim9}
For each vertex $x\in B$, there are exactly one blue edge from $x$ to $A$.
\end{claim}

\begin{proof}
Assume, to the contrary, that there exists a vertex $x\in B$ such
that the edges from $x$ to $A$ are red or there are two blue edges
from $x$ to $A$. In the former case, we can find a red copy of $S(n,m)$ centered at vertices in $H_{1}$ and $B$, a
contradiction. In the latter case, we can find a blue copy of $S(n,m)$ centered at vertices in $H_{3}$ and $B$, a
contradiction.
\end{proof}

By Claim~\ref{Claim9}, we may choose a blue edge $uv$ such that $u\in A$ and $v\in B$. Since $v$ has exactly one blue edge to $A$, this means that $u$ is the only vertex in its part, say $u\in H_t$ with $H_t=\{u\}$. Furthermore, the edges from $u$ to either $H_1$ or $H_3$ must be red to avoid making a blue copy of $S(n, m)$. Then, as before, $u$ has at least $n+m+1$ incident edges in either red or blue. If the number of red
edges incident to $u$ is at least $n+m+1$, then we can find a red copy of $S(n, m)$ centered at $v$ and a vertex $v$ in $H_1$ or $H_3$ (with a red edge to $u$), a contradiction. If
the number of blue edges incident to $u$ is at least $n+m+1$, then we can find a blue copy of $S(n, m)$ centered at $v$ and a vertex $v$ in $H_2$, a
contradiction.

\begin{case}
$t \geq 4$ and $r=0$.
\end{case}

Since all parts are small, each has order at most $m$. Let $A$ be the set of parts with blue edges to $H_1$ and $B$ be the set of parts with red edges to $H_1$. Without loss of generality, suppose that $|A|\geq |B|$. Since $|A|+|B|\geq 2n+2m+2$, we assume that $|A|\geq n+m+1$. In order to avoid a blue copy of $S(n,m)$, there exists two parts, say $H_i$ and $H_j$ within $A$ such that the edges from $H_i$ to $H_j$ are red. Choose $u\in H_i$ and $v\in H_j$. Then there are at most $m$ blue edges incident to $u$ within $A \cup B$, and there are at most $m$ blue edges incident to $v$ within $A \cup B$. Furthermore, there are at most $m-1$ green edges incident to each of $u$ and $v$ (within the parts). This means that there are at least $2n$ red edges incident to $u$, and there are at least $2n$ red edges incident to $v$. These red edges allow us to create a red copy of $S(n,m)$, a contradiction.

\begin{case}
$t \geq 4$ and $r=1$.
\end{case}

Let $A$ be the set of parts with blue edges to $H_1$ and $B$ be the set of parts with red edges to $H_1$. If $|H_{1}|\geq n+1$, then to avoid a red or blue copy of $K_{m + 1, n + 1}$ (and therefore $S(n,m)$), we must have $|A|\leq m$ and $|B|\leq m$, and hence $|H_{1}|\geq 2n+m$. Note that $A$ and $B$ are both non-empty since otherwise there would be a $2$-partition, contradicting the minimality of $t$. In this case, we get the following easy claim.

\begin{claim}\label{Claim10}
For each vertex $x\in H_{1}$, the number of red (or respectively blue) edges incident to $x$ is at most $m$.
\end{claim}

\begin{proof}
Assume, to the contrary, that there exists a vertex $v\in H_{1}$ with at least $m + 1$ incident red edges. Choose an arbitrary vertex $u\in A$. Since $uv$ is red, it follows that the edge $uv$, $m$ of the other red edges incident to $v$, and $n$ of the other red edges from $u$ to $H_1$ form a red copy of $S(n,m)$, a contradiction.
\end{proof}

From Claim~\ref{Claim10}, for each vertex $x\in H_{1}$, the number of blue or red edges incident to $x$ is at most $m$. Then the number of green edges incident to $x$ is at least $2n-m \geq n + m + 1$. Then any green edge $xy$ within $H_1$ is the center of a green copy of $S(n,m)$, a contradiction.

Thus, we may assume that $|H_{1}|\leq n$. Choose another part from the $H_{2}$ and note that $|H_{2}|\leq m$. Let $A$ be the set of parts with blue edges to $H_{2}$ and $B$ be the set of parts with red edges to $H_{2}$. Without loss of generality, suppose that $|A|\geq |B|$. Since $|A|+|B|\geq 2n+2m+2$, we assume $|A|\geq n+m+1$.

For each vertex $x \in A$, we claim that $x$ has many incident red edges. Conversely, if $x$ has at least $m + 1$ incident blue edges, then $x$ along with an arbitrarily chosen vertex in $H_{2}$ form the centers of a blue copy of $S(n, m)$, a contradiction. Thus, each vertex $x$ must have at most $m$ incident blue edges. Since green edges can exist only within parts of the Gallai partition, each vertex $x \in (A \cap H_{i})$ has at least $\ell - (m + |H_{i}|)$ incident red edges.

To avoid a blue copy of $S(n,m)$ within $A$, there must exist (at least) two parts, say $H_i$ and $H_j$, such that the edges from $H_i$ to $H_j$ are red. Choose $u\in H_i$ and $v\in H_j$. Since $r = 1$, we have $|H_i|\leq m$ or $|H_j|\leq m$ and in either case, $|H_{i}|, |H_{j}| \leq n$. Then there are at least $(n+m+1)$ red edges incident to $u$ and to $v$. Then there is a red $S(n,m)$ centered at $u$ and $v$, a contradiction.

\begin{case}
$t \geq 4$ and $r=2$.
\end{case}

From the minimality of $t$, there exists a (small) part $H_j$ with $3\leq j\leq t$ such that the edges from $H_j$ to $H_1$ are red and the edges from $H_j$ to $H_2$ are blue since otherwise $H_{1}$ and $H_{2}$ could be merged into a single part of the partition. To avoid a monochromatic copy of $K_{m+1,n+1}$ (and therefore $S(n, m)$), $|H_1|\leq n$ and $|H_2|\leq n$. Let $A$ be the set of parts with blue edges to $H_{j}$ and $B$ be the set of parts with red edges to $H_{j}$. Without loss of generality, suppose that $|A|\geq |B|$. Since $|A|+|B|\geq 2n+2m+2$, we have $|A|\geq n+m+1$. Since $A$ satisfies the same properties as in the previous case, the proof is complete.
\end{proof}

Next, we give the lower bound of $gr_{k}(K_{3} : S(n,m))$ for $k\geq 4$.

\begin{lemma}\label{Lem:RamseyLowerk4}
Let $n,m$ be two integers with $n\geq m$. For $k\geq 4$, we have
$$
gr_{k}(K_{3} : S(n,m)) \geq \begin{cases}
5\cdot\frac{n}{2} + m(k-3)+1 & \text{ if $n$ is even,}\\
5\cdot\frac{n-1}{2} + m(k-3)+ 2 & \text{ if $n$ is odd.}
\end{cases}
$$
\end{lemma}
\begin{proof}
We prove this result by inductively constructing a coloring of
$K_{\ell}$ where
$$
\ell=\begin{cases}
5\cdot\frac{n}{2} + m(k-3) & \text{ if $n$ is even,}\\
5\cdot\frac{n-1}{2} + m(k-3)+ 1 & \text{ if $n$ is odd.}
\end{cases}
$$
which contains no rainbow triangle and no monochromatic copy of
$S(n,m)$. Let $G_3$ be the complete graph of order
$$
\begin{cases}
5\cdot\frac{n}{2} & \text{ if $n$ is even,}\\
5\cdot\frac{n-1}{2} + 1 & \text{ if $n$ is odd.}
\end{cases}
$$
with three colors $1,2,3$ provided by Lemma~\ref{Lemma:k3Lower}.

Let $X_m$ be a clique of order $m$ with color $1$. For each $i$ with $4\leq i\leq k$, we let $G_i$ be a clique obtained from $G_{i-1}$ by adding a copy of $X_m$ and adding the edges from $X_i$ to $G_{i-1}$ with color $i$.

This coloring clearly contains no rainbow triangle and, since there is no vertex with degree at least $n + 1$ adjacent to a vertex with degree at least $m + 1$ in the new color, there can be no monochromatic copy of $S(n,m)$, completing the construction.
\end{proof}

We now prove the exact value of $gr_{k}(K_{3} : S(n,m))$ for $k\geq 4$.

For this result, we recall a theorem of Gy{\'a}rf{\'a}s and Simonyi.

\begin{theorem}[\cite{GyarfasSimonyi}]\label{StarThm}
In every Gallai coloring of $K_{n}$, there is a vertex with degree at least $\frac{2n}{5}$ in one color.
\end{theorem}

\begin{proposition}\label{Pro:RamseyUpperk4}
For $k\geq 4$, if $n \geq 6m + 5$, then we have
$$
gr_{k}(K_{3} : S(n,m))=\begin{cases}
5\cdot\frac{n}{2} + m(k-3)+1 & \text{ if $n$ is even,}\\
5\cdot\frac{n-1}{2} + m(k-3)+ 2 & \text{ if $n$ is odd.}
\end{cases}
$$
\end{proposition}
\begin{proof}
Let $G$ be a coloring of $K_{\ell}$ where
$$
\ell=\begin{cases}
5\cdot\frac{n}{2} + m(k-3)+1 & \text{ if $n$ is even,}\\
5\cdot\frac{n-1}{2} + m(k-3)+ 2 & \text{ if $n$ is odd}
\end{cases}
$$
containing no rainbow triangle and no monochromatic copy of $S(n, m)$.

Let $T$ be a largest set of vertices in $V(G)$ with the properties that:
\bi
\item each vertex in $T$ has all one color on its edges to $G \setminus T$, and
\item $|G \setminus T| \geq n + 1$.
\ei
For each color $i$ with $1 \leq i \leq k$, let $T_{i}$ be the set of vertices in $T$ with color $i$ on their edges to $G \setminus T$. Then in order to avoid a monochromatic copy of $K_{m + 1, n + 1}$ (and therefore a monochromatic copy of $S(n, m)$), we immediately see that $|T_{i}| \leq m$ and so $|T| \leq km$.

Let $G' = G \setminus T$. Since $|G'| \geq |G| - km \geq \frac{5n}{2} - 3m \geq n + m + 1$, no vertex in $G'$ can have at least $m$ incident edges within $G'$ in a color $i$ with $T_{i} \neq \emptyset$ to avoid a monochromatic copy of $S(n, m)$ in color $i$. More generally, we have the following fact.

\begin{fact}\label{Fact:m-1}
Each vertex of $G'$ has a total of at most $m$ incident edges in any color $i$ where $T_{i} \neq \emptyset$.
\end{fact}

By Theorem~\ref{Thm:G-Part}, there is a Gallai partition of $V(G')$, say with red and blue available for edges between the parts. By the maximality of $|T|$, there can be no single part $H_{i}$ of this Gallai partition such that $G' \setminus H_{i}$ has red edges to at most $m$ vertices of $H_{i}$ and blue edges to at most $m$ vertices of $H_{i}$.

\begin{claim}\label{Claim12}
The total number of subsets $T_{i}$ of $T$ is at most $k - 2$.
\end{claim}

\begin{proof}
Suppose, to the contrary, that there are at least $k - 1$ colors appearing on edges from $T$ to $G'$.

First suppose that there are $k$ colors appearing on edges from $T$ to $G'$. Since $|G'| \geq n + m + 1$, by Theorem~\ref{StarThm}, there is a vertex with degree at least $\frac{2(n + m + 1)}{5} \geq m$, a contradiction.

Next, suppose there are exactly $k - 1$ colors appearing on edges from $T$ to $G'$ and let red be the color not appearing. This means that $|G'| \geq 5\cdot \frac{n}{2} - 2m \geq 2n + 1$. Using Theorem~\ref{StarThm}, there is a monochromatic star but it may (and in fact must) be red. If the Gallai partition of $G'$ has only two parts, the edges between the parts must then be red. Since $|G'| \geq 2n + 1$, one of these parts has order at least $n + 1$, meaning that the other part must have order at most $m$ to avoid making a monochromatic copy of $K_{m + 1, n + 1}$. We can then move the smaller part to $T$, contradicting the maximality of $|T|$, a contradiction. Thus, we may assume that the Gallai partition of $G'$ has at least $4$ parts and two colors appearing on the edges in between the parts. More specifically, since the partition uses the smallest number of parts, both colors must appear on edges incident to vertices in each part. Perhaps one of these colors is red but let blue be one of these colors that is not red. Since blue appears on edges between $T$ and $G'$, no vertex in $G'$ can have more than $m$ incident blue edges, meaning that all parts of the partition of $G'$ have order at most $m - 1$. More generally, every vertex in $G'$ has at most $m$ incident blue edges, at most $m - 2$ incident edges of colors other than blue or red (inside the parts of the Gallai partition), and so at least $|G'| - 2m \geq n + m + 2$ incident red edges. Any choice of two vertices in $G'$ with a red edge between them forms the centers of a red copy of $S(n, m)$, a contradiction to complete the proof of Claim~\ref{Claim12}.
\end{proof}

If $|G'| \geq \frac{5n}{2} + 1$, then since the proof of Proposition~\ref{Thm:RamseyS(n,m)k3} does not use edges inside the parts, we may apply the same arguments to obtain a monochromatic copy of $S(n, m)$. Thus, 
suppose $\frac{5n}{2} - m \leq |G'| \leq \frac{5n}{2}$ in particular, this means that there are $k - 2$ subsets $T_{i}$ of $T$ and red and blue, the two colors appearing between parts of the Gallai partition of $G'$, are the two colors not represented in $T$.


By Fact~\ref{Fact:m-1}, for any vertex $x\in V(G')$, the number of edges incident to $x$ with color other than red or blue is at most $m-1$. Let $v \in V(G')$ and suppose, without loss of generality, that $v$ has at least as many incident red edges as blue. Since $n\geq 6m+5$, $v$ has at least $n+m+1$ incident red edges, say with corresponding vertices $v_{1},v_{2},\dots,v_{n+m+1}$.

\begin{claim}\label{Claim13}
For any $v_{i}$ with $1\leq i\leq n+m+1$, the total number of red edges incident to $v_{i}$ is at most $m$.
\end{claim}

\begin{proof}
Assume, to the contrary, that there exists a vertex $v_j$ with at least $m + 1$ incident red edges. Then these edges along with the edge $vv_{j}$ and $n$ other (disjoint) red edges from $v$ form a red copy of $S(n, m)$, a contradiction. 
\end{proof}

From Claim~\ref{Claim13}, for any $v_{i}$ with $1\leq i\leq n+m+1$, the number of blue edges incident to $v_{i}$ in $G'$ is at least $2n+3$. Choose two such vertices, say $v_i,v_{j}$, such that the edge $v_iv_{j}$ is blue. Then there is a blue copy of $S(n,m)$ centered at $v_{i}$ and $v_{j}$, a contradiction to complete the proof.
\end{proof}

\section{Proof of Theorem~\ref{Thm:Ramseynm}}\label{Sec3}

For general $n\geq m\geq 1$, we have the following result.
\begin{proposition}\label{Prop:RamseynmLow}
Let $n,m,k$ be three integers with $n\geq m\geq 1$, $k\geq 3$. Then
$$
gr_{k}(K_{3} : S(n,m))\geq \begin{cases}
\max\{5\cdot\frac{n}{2},n+2m+1\} + m(k-3)+1 & \text{ if $n$ is even,}\\
\max\{5\cdot\frac{n-1}{2},n+2m+1\}+ m(k-3)+ 2 & \text{ if $n$ is
odd.}
\end{cases}
$$
\end{proposition}
\begin{proof}
From Theorem \ref{Lem:RamseyS(n,m)}, we have $R(S(n,m),S(n,m))\geq
n+2m+2$. Let $F_2$ be a complete graph of order $n+2m+1$ containing
neither copy of $S(n,m)$ with color $1$ nor copy of $S(n,m)$ with
color $2$.

We prove this result by inductively constructing a coloring of
$K_{\ell}$ where
$$
\ell=\begin{cases}
n+2m+1 + m(k-3) & \text{ if $n$ is even,}\\
n+2m+1+ m(k-3)+ 1 & \text{ if $n$ is odd.}
\end{cases}
$$
which contains no rainbow triangle and no monochromatic copy of
$S(n,m)$.

Let $X$ be the clique of order $m$ colored with color $1$. For each $i$ with $3\leq i\leq k$, let $F_i$ be a clique obtained from $|F_{i-1}|$ by adding a copy of $X$ and adding the edges from $X$ to $F_{i-1}$ with color $i$. This coloring clearly contains no rainbow triangle and no monochromatic copy of $S(n,m)$, completing the construction.
\end{proof}

\begin{proposition}\label{Prop:RamseynmUp}
Let $n,m,k$ be three integers with $n\geq m\geq 1$, $k\geq 3$. If
$n\leq 6m+6$, then
$$
gr_{k}(K_{3} : S(n, m))\leq \begin{cases}
2n+m(k+3)+8 & \text{ if $n$ is even,}\\
2n+m(k+3)+9 & \text{ if $n$ is odd.}
\end{cases}
$$
\end{proposition}
\begin{proof}
Let $G$ be a coloring of $K_\ell$ where
$$
\ell=\begin{cases}
2n+m(k+3)+8 & \text{ if $n$ is even,}\\
2n+m(k+3)+9 & \text{ if $n$ is odd.}
\end{cases}
$$

First suppose that $k=3$ so $G$ is a coloring of $K_\ell$ where
$$
\ell=\begin{cases}
2n+6m+8 & \text{ if $n$ is even,}\\
2n+6m+9& \text{ if $n$ is odd}
\end{cases}
$$
with no rainbow triangle and suppose, for a contradiction, that $G$ contains no monochromatic copy of $S(n, m)$. Since $G$ contains no rainbow triangle, it follows from Theorem~\ref{Thm:G-Part} that there is a Gallai partition of $V(G)$. Suppose red and blue are the two colors appearing in this partition and let green be the third available color. Let $H_1,H_2,\ldots,H_t$ be the parts in this partition and choose such a partition so that $t$ is minimized.

\begin{claim}\label{Claim16}
For each part $H_i$ of $K_{\ell}$, $|H_i|\leq m$ or $|H_i|\geq n+1$.
\end{claim}
\begin{proof}
Suppose that there exists one part, say $H_j$, and $m+1\leq |H_j|\leq n$. Let $A$ be the set of parts with red edges to $H_j$, and $B$ be the set of parts with blue edges to $H_j$. Without loss of generality, suppose $|A|\geq |B|$. Since $|A|+|B|\geq (2n + 6m + 8) - n \geq 2n+1$, we may assume that $|A|\geq n+1$. Then the red edges between $A$ and $H$ yield a red copy of $K_{n+1,m+1}$, which contains a red copy of $S(n. m)$, a contradiction.
\end{proof}

If $2\leq t\leq 3$, then by the minimality of $t$, we may assume $t=2$, say with corresponding parts $H_1$ and $H_2$. Suppose the edges between $H_1$ and $H_2$ are red. Without loss of generality, we assume $|H_1|\leq |H_2|$. If $|H_1|\geq n+1$, then there is a red monochromatic copy of $S(n,m)$, a contradiction. From Claim~\ref{Claim16}, we have $|H_1|\leq m$. To avoid a monochromatic copy of $S(n,m)$, for each $x\in H_2$, the total number of red edges incident to $x$ is at most $m$. Then the number of blue and green edges incident to $x$ is at least $2n+5m+7$. Without loss of generality, we assume that the number of blue edges incident to $x$ is at least $\frac{2n+5m+7}{2}\geq n+2m+1$. Let $x_1,x_2,\ldots,x_{n+2m+1}$ be neighbors of $x$ by these blue edges. In order to avoid creating a blue copy of $S(n,m)$, for each $x_i$, the total number of blue edges incident to $x_i$ is at most $m$, and hence the number of green edges incident to $x_i$ is at least $2n+4m+7$. We can then easily find a green copy of $S(n,m)$ centered at two of these vertices, a contradiction.

Thus, we may assume that $t\geq 4$. Let $r$ be the number of parts of the Gallai partition so $|H_i|\geq n+1$ for each $i$ with $1\leq i\leq r$. Call such parts large and all remaining parts (of order at most $n$, and therefore at most $m$ by Claim~\ref{Claim16}) small. In order to avoid a monochromatic copy of $S(n,m)$, we immediately have $r\leq 1$.

\setcounter{case}{0}
\begin{case}
$r=0$.
\end{case}

By Claim~\ref{Claim16}, every part has order at most $m$. Let $A$ be the set of parts with blue edges to $H_1$ and $B$ be the set of parts with red edges to $H_1$. Without loss of generality, suppose that $|A|\geq |B|$. Since $|A|+|B|\geq 2n+5m+8$, we may assume that $|A|\geq n+2m+4$. In order to avoid a blue copy of $S(n,m)$, there exist two parts within $A$, say $H_i,H_j$, such that the edges from $H_i$ to $H_j$ are red. Choose $u\in H_i$ and $v\in H_j$.
There are at most $m$ blue edges incident to $u$ and at most $m$ blue edges incident to $v$ since otherwise one of these vertices along with a vertex of $H_{1}$ would form the centers of a blue copy of $S(n, m)$. Furthermore, there are at most $m-1$ green edges incident to $u$ and $v$, and so there are at least $2n+4m+7$ red edges incident to $u$ and at least $2n+4m+7$ red edges incident to $v$. This means there is a red copy of $S(n,m)$ centered at $u$ and $v$, a contradiction.

\begin{case}
$r=1$.
\end{case}

Let $H_{1}$ be the (unique) large part and let $A$ be the set of parts with blue edges to $H_1$ and $B$ be the set of parts with red edges to $H_1$. Since $|H_{1}|\geq n+1$, in order to avoid a red or blue copy of $S(n,m)$, we have $|A|\leq m$ and $|B|\leq m$, and hence $|H_{1}|\geq 2n+4m+8$. Since $A$ and $B$ are both nonempty (since $t \geq 4$), the following fact is immediate.

\begin{fact}\label{Fact:17}
For each vertex $x\in H_{1}$, the total number of blue or red edges incident to $x$ is at most $m$.
\end{fact}

From Fact~\ref{Fact:17}, for each vertex $x\in H_{1}$, the number of blue or red edges incident to $x$ is at most $m$. Thus, the number of green edges incident to $x$ is at least $2n+4m+7$. Choosing any green edge $xy$ within $H_{1}$, this edge forms the center of a green copy of $S(n,m)$, a contradiction.

Now we may assume that $k\geq 4$.

As in the proof of Proposition~\ref{Pro:RamseyUpperk4}, let $T$ be a largest set of vertices in $V(G)$ with the properties that:
\bi
\item each vertex in $T$ has all one color on its edges to $G \setminus T$, and
\item $|G \setminus T| \geq n + 1$.
\ei
For each color $i$ with $1 \leq i \leq k$, let $T_{i}$ be the set of vertices in $T$ with color $i$ on their edges to $G \setminus T$. Then in order to avoid a monochromatic copy of $K_{m + 1, n + 1}$ (and therefore a monochromatic copy of $S(n, m)$), we immediately see that $|T_{i}| \leq m$ and so $|T| \leq km$.

Let $G' = G \setminus T$. Since 
$$
|G'| \geq \ell - km \geq [2n + m(k + 3) + 8] - km \geq 2n + 3m + 8,
$$
no vertex in $G'$ can have at least $m$ incident edges within $G'$ in any color $i$ with $T_{i} \neq \emptyset$ to avoid a monochromatic copy of $S(n, m)$ in color $i$. More generally, we have the following fact.

\begin{fact}\label{Fact:m-1gen}
Each vertex of $G'$ has a total of at most $m$ incident edges in any color $i$ where $T_{i} \neq \emptyset$.
\end{fact}

By Theorem~\ref{Thm:G-Part}, there is a Gallai partition of $V(G')$, say with red and blue available for edges between the parts. By the maximality of $|T|$, there can be no single part $H_{i}$ of this Gallai partition such that $G' \setminus H_{i}$ has red edges to at most $m$ vertices of $H_{i}$ and blue edges to at most $m$ vertices of $H_{i}$. This leads to the following analogue of Claim~\ref{Claim12}

\begin{claim}\label{Claim12gen}
The total number of subsets $T_{i}$ of $T$ is at most $k - 2$.
\end{claim}

\begin{proof}
Suppose, to the contrary, that there are at least $k - 1$ colors appearing on edges from $T$ to $G'$.

First suppose that there are $k$ colors appearing on edges from $T$ to $G'$. Since $|G'| \geq 2n + 3m + 8$, by Theorem~\ref{StarThm}, there is a vertex with degree at least $\frac{2(2n + 3m + 8)}{5} \geq m$, a contradiction to Fact~\ref{Fact:m-1gen}.

Next, suppose there are exactly $k - 1$ colors appearing on edges from $T$ to $G'$ and let red be the color not appearing. This means that $|G'| \geq 2n + 4m + 8 \geq 2n + 1$. Using Theorem~\ref{StarThm}, there is a monochromatic star but it may (and in fact must) be red. If the Gallai partition of $G'$ has only two parts, the edges between the parts must then be red. Since $|G'| \geq 2n + 1$, one of these parts has order at least $n + 1$, meaning that the other part must have order at most $m$ to avoid making a monochromatic copy of $K_{m + 1, n + 1}$. We can then move the smaller part to $T$, contradicting the maximality of $|T|$, a contradiction. Thus, we may assume that the Gallai partition of $G'$ has at least $4$ parts and two colors appearing on the edges in between the parts. More specifically, since the partition uses the smallest number of parts, both colors must appear on edges incident to vertices in each part. Perhaps one of these colors is red but let blue be one of these colors that is not red. Since blue appears on edges between $T$ and $G'$, no vertex in $G'$ can have more than $m$ incident blue edges, meaning that all parts of the partition of $G'$ have order at most $m - 1$. More generally, within $G'$, every vertex in $G'$ has at most $m-1$ incident blue edges, at most $m - 2$ incident edges of colors other than blue or red (inside the parts of the Gallai partition), and so at least $|G'| - (2m - 2) \geq 2n + 2m + 10$ incident red edges. Any choice of two vertices in $G'$ with a red edge between them forms the centers of a red copy of $S(n, m)$, a contradiction to complete the proof of Claim~\ref{Claim12gen}.
\end{proof}

Let $v$ be a vertex of $G'$ and suppose $v$ has at least as many incident red edges (within $G'$) as blue edges. Then $v$ has at least $\frac{\ell - (k - 2)m}{2} \geq n + 2m + 4$ incident red edges, say with corresponding vertices $v_{1},v_{2},\dots,v_{n+2m+4}$. The following fact is an easy analogue of Claim~\ref{Claim13}.

\begin{fact}\label{Fact13}
For any $v_{i}$ with $1\leq i\leq n+m+1$, the total number of red edges incident to $v_{i}$ is at most $m$.
\end{fact}

From Fact~\ref{Fact13}, for any $v_{i}$ with $1\leq i\leq n+2m+4$, the number of blue edges incident to $v_{i}$ in $G'$ is at least $2n+4m+ 8$. Choose two such vertices, say $v_i,v_{j}$, such that the edge $v_iv_{j}$ is blue. Then there is a blue copy of $S(n,m)$ centered at $v_{i}$ and $v_{j}$, a contradiction to complete the proof.
\end{proof}

\end{document}